\documentclass[a4paper]{amsart}

\usepackage{amssymb,latexsym}
\usepackage{amsmath}
\usepackage{amsfonts}
\usepackage{amsthm}
\usepackage{amssymb, comment}

\theoremstyle{plain}
\newtheorem{theorem}{Theorem}

\newtheorem{proposition}[theorem]{Proposition}

\newtheorem{question}{Question}

\def\Q{\mathbb{Q}}

\theoremstyle{definition}
\newtheorem{definition}{Definition}

\newtheorem{remark}{Remark}

\title[Rational Diophantine sextuples with strong pair]
{Rational Diophantine sextuples with strong pair}
\begin{document}

\date{}


\author[A. Dujella]{Andrej Dujella}
\address{
Department of Mathematics\\
Faculty of Science\\
University of Zagreb\\
Bijeni{\v c}ka cesta 30, 10000 Zagreb, Croatia
}
\email[A. Dujella]{duje@math.hr}

\author[M. Kazalicki]{Matija Kazalicki}
\address{
Department of Mathematics\\
Faculty of Science\\
University of Zagreb\\
Bijeni{\v c}ka cesta 30, 10000 Zagreb, Croatia
}
\email[M. Kazalicki]{matija.kazalicki@math.hr}

\author[V. Petri\v{c}evi\'c]{Vinko Petri\v{c}evi\'c}
\address{
Faculty of Electrical Engineering, Computer Science and Information Technology Osijek\\
Josip Juraj Strossmayer University of Osijek\\
Kneza Trpimira 2B, HR-31000 Osijek, Croatia
}
\email[V. Petri\v{c}evi\'c]{vinko.petricevic@ferit.hr}

\begin{abstract}
A set of $m$ distinct nonzero rationals $\{a_1, a_2,\ldots, a_m\}$ such that $a_i a_j+1$ is a perfect square for all $1\le i <j \le m$, is called a rational Diophantine $m$-tuple. If in addition, $a_i^2+1$ is a perfect square for $1\le i\le m$, then we say the $m$-tuple is strong. In this paper, we construct infinite families of rational Diophantine sextuples containing a strong Diophantine pair.

\end{abstract}


\maketitle

\section{Introduction}

A Diophantine $m$-tuple is a set of $m$ distinct positive
integers with the property that the product of any two of its distinct
elements plus $1$ is a square. Fermat found the first Diophantine quadruple in integers $\{1,3,8,120\}$.
If a set of $m$ nonzero rationals
has the same property, then it is called
a rational Diophantine $m$-tuple. If in addition, a rational Diophantine $m$-tuple has the property that the square of each element plus $1$ is a square, we say that it is {\bf strong}.
The first example of a rational Diophantine quadruple was the set
$$
\left\{\frac{1}{16},\, \frac{33}{16},\, \frac{17}{4},\, \frac{105}{16}\right\}
$$
found by Diophantus. Euler proved that the exist
infinitely many rational Diophantine quintuples (see \cite{Hea}),
in particular he was able to extend the integer Diophantine quadruple found by Fermat,
to the rational quintuple
$$
\left\{ 1, 3, 8, 120, \frac{777480}{8288641} \right\}.
$$
Stoll \cite{Stoll} recently showed that this extension is unique.
Therefore, the Fermat set $\{1,3,8,120\}$ cannot be extended to a rational Diophantine sextuple.

In 1969, using linear forms in logarithms of algebraic numbers and a reduction method
based on continued fractions, Baker and Davenport \cite{B-D}
proved that if $d$ is a positive integer such that
$\{1, 3, 8, d\}$ forms a Diophantine quadruple, then $d$ has to be $120$.
This result motivated the conjecture that there does not exist a Diophantine quintuples in integers.
The conjecture has been proved recently by He, Togb\'e and Ziegler \cite{HTZ}
(see also \cite{duje-crelle}).

In the other hand, it is not known how large can be a rational Diophantine tuple.
In 1999, Gibbs found the first example of rational Diophantine sextuple \cite{Gibbs1}
$$
\left\{ \frac{11}{192}, \frac{35}{192}, \frac{155}{27}, \frac{512}{27}, \frac{1235}{48}, \frac{180873}{16} \right\}.
$$
In 2017 Dujella, Kazalicki, Miki\'c and Szikszai \cite{DKMS} proved that there are
infinitely many rational Diophantine triples that can be extended to a Diophantine sextuple in infinitely many ways, while
Dujella and Kazalicki \cite{Duje-Matija} (inspired by the work of Piezas \cite{P})
described another construction of parametric families of rational Diophantine sextuples.
Dujella, Kazalicki and Petri\v{c}evi\'c \cite{DKP-sext}
proved that there are infinitely many rational Diophantine sextuples
such that denominators of all the elements (in the lowest terms) in the sextuples are perfect squares, and also proved \cite{DKP-reg} that there are infinitely many rational Diophantine sextuples containing two regular quadruples and one regular quintuple.
No example of a rational Diophantine septuple is known.
Lang's conjecture on varieties of general type implies that
the number of elements in a rational Diophantine tuple is bounded by an absolute constant
(for more details, see the introduction of \cite{DKMS}).
For additional information on Diophantine $m$-tuples, refer to the survey article \cite{Duje-notices} and the book \cite{duje-book}.

In this paper, we study rational Diophantine sextuples which contain a strong elements (i.e. the elements $a$ with the property that $a^2+1$ is a perfect square).

Denote by $C$ an affine curve given by the equation $p(u,v)=0$ where
\begin{align*}
	p(u, v)  = & \,\, 3u^4v^4 - 8u^4v^3 + 6u^4v^2 - u^4 \\
	& - 8u^3v^4 + 4u^3v^3 - 8u^3v^2 + 12u^3v + 6u^2v^4 \\
	& - 8u^2v^3 + 4u^2v^2 + 8u^2v + 6u^2 + 12uv^3 + 8uv^2 \\
	& + 4uv + 8u - v^4 + 6v^2 + 8v + 3.
\end{align*}

The curve $C$ is birationally equivalent to the elliptic curve $$E:\quad y^2 + xy + y = x^3 - 33x + 68.$$ Torsion subgroup of Mordell-Weil group of $E/\Q$ is generated by the point $[-1,10]$ of order $6$, while the free part of the group is generated by the point $[11/4,-25/8]$. In particular, $E$ has infinitely many rational points.

Define three parametric families

\begin{align*}
	\mathcal{F}_1(u,v) &=\begin{bmatrix}
		\frac{2u}{(u-1)(u+1)}, & \frac{2v}{(v-1)(v+1)}, & \frac{2(v-1)(v+1)(u-1)(u+1)}{(-v+uv-u-1)^2}
	\end{bmatrix}, \\
\mathcal{F}_2(u,v) &=\begin{bmatrix}
		\frac{2u}{(u-1)(u+1)}, & -\frac{2(u-v)(uv+1)}{(uv+v+1-u)(uv-v+u+1)}, & -\frac{2(uv-v+u+1)(u^3v-u^3-v-1)}{(u-1)(u+1)(uv+v+1-u)^2}
	\end{bmatrix}, \\
	\mathcal{F}_3(u,v) &=\begin{bmatrix}
		-\frac{2v}{(v-1)(v+1)}, & -\frac{2(u-v)(uv+1)}{(uv+v+1-u)(uv-v+u+1)}, & \frac{2(uv+v+1-u)(v^3u-v^3-u-1)}{(uv-v+u+1)^2(v+1)(v-1)}
	\end{bmatrix}.
\end{align*}

By carefully selecting parameters $(u,v)$, we can utilize methods described in \cite{DKMS} to extend Diophantine triples to Diophantine sextuples, thus deriving our main result.

\begin{theorem}\label{thm:1} If $(u,v)\in C(\Q)$, then each triple $\mathcal{F}_i(u,v)$ is a rational Diophantine triple (provided that all the elements are defined, distinct and nonzero), whose first two elements form a strong Diophantine pair. Moreover, each such $\mathcal{F}_i(u,v)$ can be extended to a rational Diophantine sextuple in infinitely many ways.
	
\end{theorem}

\begin{remark}\label{rem:1}
	Note that $\mathcal{F}_2(v,u)=-\mathcal{F}_3(u,v)=\mathcal{F}_3(-u,1/v)$ for all pairs $(u,v)$. Therefore, since the mappings $(u,v)\mapsto (v,u)$ and $(u,v)\mapsto (-u,1/v)$ are the automorphisms of the curve $\mathcal{C}$, the families $\mathcal{F}_2$ and $\mathcal{F}_3$ are parameterizing the same sets of triples.
\end{remark}
As a corollary, we obtain the following result.

\begin{theorem}
	There are infinitely many rational Diophantine sextuples that contain a strong Diophantine pair.
\end{theorem}

\section{Induced elliptic curves and overview of \cite{DKMS}}\label{sec:induced}

To extend a rational Diophantine triple $\{a,b,c\}$ to a quadruple, we need to find $d\in \Q$ for which
$ad+1, bd+1$ and $cd+1$ are perfect squares. Such $d$ naturally defines a rational point on the elliptic curve $y^2=(ax+1)(bx+1)(cx+1)$ which
is isomorphic (via transformation $x\mapsto x/abc, y\mapsto y/abc$) to the curve
$$E_{a,b,c}:\quad y^2=(x+ab)(x+ac)(x+bc).$$
Conversely, the two descent argument implies that each $d$ is equal to $x(T+P)/abc$ for some $T\in 2E_{a,b,c}(\Q)$ and $P=[0,abc]\in E_{a,b,c}(\Q) $ (see Proposition 1 in \cite{duje-bordo}).

Besides the rational points of order $2$, $$T_1=[-ab,0],\quad T_2=[-ac,0], \quad T_3=[-bc,0],$$ we will also need rational point
$S=[1,rst]\in E_{a,b,c}(\Q),$
where $ab+1=r^2$, $ac+1=s^2$ and $bc+1=t^2$, for some $r,s,t \in \Q$.
Note that $S=2R$, where $R=[rs+rt+st,(r+s)(r+t)(s+t)]$. In the case when $\{a,b\}$ is a strong pair, we have two more rational points
$$A=[a\cdot abc,abc\cdot rsu], \quad B=[b \cdot abc,abc \cdot rtv]\in E_{a,b,c}(\Q),$$
where  $a^2+1=u^2$ and $b^2+1=v^2$ for some $u,v \in \Q$.

The main result of \cite{DKMS} states that if $\{a,b,c\}$ is a rational Diophantine triple such that the point $S$ on induced elliptic curve $E_{a,b,c}$ has order $3$, then for each integer~$n$
$$\left\{a,b,c,\frac{x([2n+1]P)}{abc},\frac{x([2n+1]P+S)}{abc},\frac{x([2n+1]P-S)}{abc}\right\}$$
is a rational Diophantine sextuple. Moreover, Lemma 1 in \cite{DKMS} shows that the order of $S$ is $3$ if and only if $S(a,b,c)=0$ where
$$S(a,b,c)=3+4(ab+ac+bc)+6abc(a+b+c)-(abc)^2(-12+a^2 + b^2 + c^2 - 2ab - 2ac - 2bc).$$

Thus we are led to the following question.
\begin{question}
	Are there infinitely many rational Diophantine triples $\{a,b,c\}$ for which $a^2+1$ and $b^2+1$ are perfect squares and $S(a,b,c)=0$?
	We refer to such triples as {\bf special}.
\end{question}

For an affirmative answer to this question, one would need to find a curve of genus zero or one (with infinitely many rational points) on the surface of the general type, which is a $32$-cover of the surface $S(a,b,c)=0$. This surface is defined by the condition that $ab+1$, $ac+1$, $bc+1$, $a^2+1$, and $b^2+1$ are perfect squares. In general, this is a difficult problem, so we sought inspiration from experimental data.

\section{Experiments and regularity} \label{sec:experiments}

Our key insight came from examining numerical examples of special Diophantine triples

\begin{center}
	\{30464/2223, 22815/5168, 361/7956\},\\
	\{30464/2223, 4807/31824, 10881/1292\},\\
	\{-22815/5168, 4807/31824, -8092/2223\}.
\end{center}	

To understand these examples, it is necessary to introduce the concept of regularity (see \cite{DKP-reg,DP2}).

\begin{definition}
	The quadruple $(a,b,c,d)\in \Q^4$ is called {\bf regular} if $r_4(a,b,c,d)=0$ where
	$$r_4(a,b,c,d)=(a+b-c-d)^2-4(ab+1)(cd+1).$$ Similarly, the quintuple $(a,b,c,d,e)$ is regular if $r_5(a,b,c,d,e)=0$
	where $$r_5(a,b,c,d,e)=(abcde+2abc+a+b+c-d-e)^2-4(ab+1)(ac+1)(bc+1)(de+1).$$ 
	Note that polynomials $r_4$ and $r_5$ are symmetric.
\end{definition}

In the examples above, we noticed that for the first triple $\{a,b,c\}$ the (improper) quintuple $\{a,a,b,b,c\}$ is regular, i.e. $r_5(a,a,b,b,c)=0$.
Similarly, for the second and third triple the (improper) quadruple $\{a,b,b,c\}$ is regular, i.e. $r_4(a,b,b,c)=0$. Furthermore, the elliptic curves associated to these Diophantine triples are isomorphic to each other.

These regularity conditions can be restated in the context of the arithmetic of the elliptic curve $E_{a,b,c}$.

\begin{proposition}\label{prop:regularity}
	Let $\{a,b,c\}$ be a rational Diophantine triple containing a strong pair $\{a,b\}$. Let $A$, $B$, $P$, and $S$ be points in $E_{a,b,c}(\Q)$ as defined in Section \ref{sec:induced}.
	We have that
	\begin{enumerate}
		\item [a)] $r_4(a,a,b,c)=0$ if and only if $A= \pm P \pm S$ for some choice of signs,
		\item [b)] $r_5(a,a,b,b,c)=0$ if and only if $A\pm B \pm S=\mathcal{O}$ for some choice of signs.
	\end{enumerate}
\end{proposition}
\begin{proof}
It is known (see Section 3.1 of \cite{duje-book}) that for a Diophantine triple $\{a,b,c\}$, $r_4(a,b,c,d)=0$ if and only if $d=x(P\pm S)$, or equivalently $D=\pm P \pm S$ for some choice of signs, where $D\in E_{a,b,c}(\Q)$ and $x(D)=d$. Similarly, for a Diophantine quintuple $\{a,b,c,d\}$, $r_5(a,b,c,d,e)=0$ if and only if $e=x(D \pm S)$ or equivalently $E=\pm D \pm S$ for some choice of signs, where $E\in E_{a,b,c}(\Q)$ and $x(E)=e$.

Both claims follow when we apply these results to $E_{a,b,c}$ and points $D=A$ and $E=B$.

\end{proof}

\section{Proof of Theorem \ref{thm:1}}

To construct family $\mathcal{F}_1$, we proceed as follows. Set $a=\frac{2u}{u^2-1}$ and $b=\frac{2v}{v^2-1}$ to ensure that $a^2+1$ and $b^2+1$ are perfect squares. If we substitute these values in $$r_5(a,a,b,b,c)=(abc)^2-2a c^2 b-4 a c+c^2-4cb-4$$ the resulting expression factors as $r_5(a,a,b,b,c)=q_1 q_2$ where
$$q_1=u^2v^2c + 2ucv^2 + 2cvu^2 + cv^2 - 2cv + c - 2uc + cu^2 + 2 - 2v^2 - 2u^2 + 2u^2v^2,$$
$$q_2=cv^2 - 2ucv^2 + 2cv + u^2v^2c - 2cvu^2 + cu^2 + 2uc + c - 2 + 2v^2 - 2u^2v^2 + 2u^2.$$
Solving for $c$ in $q_2=0$ we obtain two solution one of which is
$$c=\frac{2(u^2v^2 - u^2 - v^2 + 1)}{(-v + uv - u - 1)^2}.$$ If we substitute all this in $S(a,b,c)=0$, the expression factors as $s_1s_2s_3$ where
\begin{align*}
	s_1&=1 + 8vu^4 - 8u^3v^2 - 8v^3u^2 + 4vu^3 + 8uv^2 - 8v^3 + 8vu^2 \\
	&+ 8uv^4 + 12v^3u^3 + 4uv^3 + 12uv - 4u^2v^2 - 6u^2v^4 + u^4v^4 \\
	&- 6u^4v^2 - 6u^2 - 6v^2 - 3v^4 - 3u^4 - 8u^3,\\
	s_2&=3 - 8u^3v^2 + 8u - 8v^3u^2 + 12vu^3 + 8v - 8v^3u^4 \\
	&- 8u^3v^4 + 8uv^2 + 8vu^2 + 4v^3u^3 + 12uv^3 + 4uv \\
	&+ 4u^2v^2 + 6u^2v^4 + 3u^4v^4 + 6u^4v^2 + 6u^2 + 6v^2 - v^4 - u^4,\\
    s_3&=(uv + v - u + 1)^2 (-v + uv + u + 1)^2.
\end{align*}

Note that factor $s_2$ is equal to $p(u,v)$ from the definition of curve $C: p(u,v)=0$, thus given a rational point $(u,v)$ on $C$,
we obtain the triple $\mathcal{F}_1(u,v)$ from the introduction. The curve defined by $s_1=0$ is isomorphic to $C$.

It remains to show that $\{a,b,c\}$ is a Diophantine triple (note that a priori we
only know that $a^2+1$ and $b^2+1$ are perfect squares). To this end, it is important to notice that for regular quintuple $\{a,b,c,d,e\}$, not necessary Diophantine,
we have that $(ab+1)(ac+1)(bc+1)(de+1)$ is a perfect square for every permutation of elements (since polynomial $r_5(a,b,c,d,e)$ is symmetric). In particular, the regularity of $\{a,a,b,b,c\}$ implies that $a^2+1, b^2+1, ac+1$ and $bc+1$ represent the same class modulo squares (i.e. they are equal in $\Q^\times/{\Q^\times}^2$).
Since by construction $a^2+1$ is a perfect square, it remains to prove that $ab+1$  is a perfect square.

Let $t(u,v)$ denote the product of the denominator and numerator of $ab+1$. Thus, we have
\[
\begin{aligned}
	t(u,v) &= u^4v^4 - 2u^4v^2 + u^4 + 4u^3v^3 - 4u^3v - 2u^2v^4 \\
	&\quad+ 4u^2v^2 - 2u^2 - 4uv^3 + 4uv + v^4 - 2v^2 + 1.
\end{aligned}
\]

It is straightforward to verify that
$$p(u,v)+t(u,v)=(uv+1)^2(uv-u-v-1)^2,$$
hence $t(u,v)$ is a perfect square (as is $ab+1$) whenever $p(u,v)=0$. Consequently,  the conclusion of Theorem \ref{thm:1} for $\mathcal{F}_1(u,v)$ follows.

The curve given by the equation $s_1(u,v)=0$ is isomorphic to the curve $\mathcal{C}$ via the mapping $\sigma:(u,v)\mapsto (\frac{1}{u},-v)$. Since $\sigma(a)=-a$, $\sigma(b)=-b$, and $\sigma(c)=-c$, we observe that employing a parametrization by the equation $s_1(u,v)=0$ yields the same family of triples. Similarly, since the surface $q_1(u,v,c)=0$ is isomorphic to the surface $q_2(u,v,c)=0$ via the mapping $(u,v,c)\mapsto (-u,-v,c)$, it follows that we do not get anything new by employing parametrization for $c$ given by condition $q_1=0$. It is straightforward to verify that the condition $s_3(u,v)=0$ leads to triples with repeated elements. Thus, we conclude that every special rational Diophantine triple $\{a,b,c\}$ satisfying $r_5(a,a,b,b,c)=0$ belongs to the family $\mathcal{F}_1$.

Similarly, to obtain the family $\mathcal{F}_2(u,v)$ in the regularity condition
\begin{equation}\label{eg:reg}
r_4(a,a,b,c)=-4 - 4ab + b^2 - 4ac - 2bc - 4a^2bc + c^2=0,
\end{equation}
we substitute $a=\frac{2u}{u^2-1}$ and $b=\frac{2v}{v^2-1}$,  yielding the condition  $r_1 r_2=0$ where
\begin{align*}
	r_1 &= -2 - c + 2cu + 2u^2 - cu^2 - 2v - 4uv - 2u^2v + 2v^2 + cv^2 - 2cuv^2 - 2u^2v^2 + cu^2v^2, \\
	r_2 &= 2 - c - 2cu - 2u^2 - cu^2 - 2v + 4uv - 2u^2v - 2v^2 + cv^2 + 2cuv^2 + 2u^2v^2 + cu^2v^2.
\end{align*}

By solving for $c$ in the equation $r_1(u,v,c)=0$ and substituting the result into $S(a,b,c)$, we obtain $S(a,b,c)=t_1 t_2 t_3=0$, where
\begin{align*}
	t_1 &= (1 + u - v + uv)^2 (1 - u + v + uv)^2, \\
	t_2 &= -3 + 8u - 6u^2 + u^4 - 16v + 4uv + 16u^2v - 4u^3v - 10v^2 \\
	& \quad - 48uv^2 - 4u^2v^2 - 2u^4v^2 + 16v^3 - 4uv^3 - 16u^2v^3 \\
	& \quad + 4u^3v^3 - 3v^4 + 8uv^4 - 6u^2v^4 + u^4v^4, \\
	t_3 &= -1 + 6u^2 - 8u^3 + 3u^4 - 4uv + 16u^2v + 4u^3v - 16u^4v \\
	& \quad + 2v^2 + 4u^2v^2 + 48u^3v^2 + 10u^4v^2 + 4uv^3 - 16u^2v^3 \\
	& \quad - 4u^3v^3 + 16u^4v^3 - v^4 + 6u^2v^4 - 8u^3v^4 + 3u^4v^4.
\end{align*}

In this manner, we obtain a triple ${a(u,v), b(u,v), c(u,v)}$ parametrized by points $(u,v)$ on the curve $\mathcal{D}:t_3(u,v)=0$. Note that the curve $\mathcal{D}$ is isomorphic to $\mathcal{C}$ through the mapping $\alpha: \mathcal{C} \rightarrow \mathcal{D}$, defined as $(u,v) \mapsto \left(\frac{-1+uv}{u+v}, -v\right)$. By precomposing the above parametrization with the map $\alpha$, we obtain the family $\mathcal{F}_2$.

It remains to show that $\mathcal{F}_2(u,v)$ is Diophantine triple. In general, the regularity condition $r_4(a,b,c,d)=0$ implies that $(ab+1)(cd+1)$ is a perfect square for all permutation of elements, as $r_4$ is symmetric polynomial. Thus, after combining the condition $r_4(a,a,b,c)=0$ with the requirement that $a^2+1$ is a perfect square, the remaining task is to establish that $ab+1$ (or equivalently $ac+1$) is also a perfect square. This is accomplished similarly to the case of the family $\mathcal{F}_1$. Similarly to before, we deduce that any special rational Diophantine triple $\{a,b,c\}$ satisfying $r_4(a,a,b,c)=0$ belongs to the family $\mathcal{F}_2$.

The statement for the family $\mathcal{F}_3$ follows from the observation that $\mathcal{F}_2(v,u)=\mathcal{F}_3(-u,1/v)$ as noted in Remark \ref{rem:1}. It follows from a discusion in Section \ref{sec:induced} that each of the triples from these families can be extended in infinitely many ways to a Diophantine sextuple.

It is intriguing that triples satisfying different regularity conditions are parameterized by the same curve. This implies that there could be a direct relationship between these families.

The observation that elliptic curves associated with the triples $\mathcal{F}_i(u,v)$, for $i=1,2,3$, are isomorphic to each other provides an answer to this question.

\section{Diophantine triples with isomorphic elliptic curves}

Let $\{a,b,c\}$ be a rational Diophantine triple for which $S \in E_{a,b,c}(\Q)$ has order $3$ (i.e. $S(a,b,c)=0$), and let $W\in E_{a,b,c}(\Q)$, $W\ne \pm S$ and $2W \ne \mathcal{O}$, be such that $1-x(W)$ is a perfect square. Write $1-x(W)=k^2$ for some $k\in \Q^\times$. We can choose the sign of $k$ such that it is equal to the sign of $y(W)$. Consider the change of variable and its inverse
\begin{equation*}
(x,y)\mapsto \left(\frac{x}{k^2}+1-\frac{1}{k^2},\frac{y}{k^3}\right),\quad (X,Y)\mapsto \left(k^2 X+1-k^2,k^3 Y\right),
\end{equation*}
which defines an isomorphism $\phi_W: E_{a,b,c} \rightarrow \tilde{E}$ where $\tilde{E}:Y^2=(X+A)(X+B)(X+C)$ for some distinct $A,B,C \in \Q$.
Note that $X(\phi_W(W))=0$, thus $ABC$ is a perfect square and $\frac{AB}{C}=c'^2, \frac{AC}{B}=b'^2$ and $\frac{BC}{A}=a'^2$ for some $a',b',c' \in \Q^\times$. We can choose signs of $a',b'$ and $c'$ such that $a'b'=C$, $a'c'=B$ and $b'c'=A$. It follows that $\tilde{E}=E_{a',b',c'}$. Since $X(\phi_W(S))=1$, and $\phi_W(S) \in 2 E_{a',b',c'}(\Q)$ (since $S\in 2E_{a,b,c}(\Q)$ and $\phi_W$ is a group isomorphism), we have that that $1+A,1+B$ and $1+C$ are perfect squares. Elements $a',b'$ and $c'$ are non-zero and distinct since $A,B$ and $C$ are non-zero and distinct, therefore $\{a',b',c'\}$ is a rational Diophantine triple. Moreover, since $\phi_W(S)=\pm S'$, it follows that $S'$ has order $3$, thus $S(a',b',c')=0$.

Conversely, let $\{a',b',c'\}$ be a rational Diophantine triple for which $S(a',b',c')=0$ and let $\phi:E_{a,b,c}\rightarrow E_{a',b',c'}$ be an isomorphism.
Denote by $W=\phi^{-1}(P')$, where $P'\in E_{a',b',c'}(\Q)$ with $X(P')=0$. Since $\phi^{-1}(X,Y)=(u^2 X+v, u^3 Y)$ for some $u,v \in \Q$, it follows from $\phi^{-1}(S')=\pm S$ that
$u^2+v=1$. Since $x(W)=v$, it follows that $1-x(W)$ is a perfect square, and  $\phi=\phi_{\pm W}$. Thus, we proved the following proposition.

\begin{proposition}\label{prop: phi}
	Let $\{a,b,c\}$ be a rational Diophantine triple such that $S(a,b,c)=0$, $E_{a,b,c}$ the corresponding elliptic curve and $W\in E_{a,b,c}(\Q)$, $6W \ne \mathcal{O}$, a point for which $1-x(W)$ is a perfect square. Then $\phi_W$ defines an isomorphism between $E_{a,b,c}$ and $E_{a',b',c'}$, where $\{a',b',c'\}$ is a rational Diophantine triple, determined up to the sign, for which $S(a',b',c')=0$. Furthermore, every rational Diophantine triple $\{a',b',c'\}$ with the property that $S(a',b',c')=0$ and $E_{a',b',c'} \cong E_{a,b,c}$ can be obtained in this manner.
\end{proposition}
\begin{remark}
The condition  $1-x(W)=k^2$ is a perfect square defines a curve $$y^2=(1-k^2+ab)(1-k^2+ac)(1-k^2+bc).$$ If $rst\ne 0$ (or equivalently, if $S$ is not a point of order $2$), this curve has genus two. Consequently, in our situation, only a finite number of points $W\in E_{a,b,c}(\Q)$ satisfy the required property. The point $P=[0,abc]$ induces the identity map.
\end{remark}

For specificity, we will select elements $a', b'$, and $c'$ such that $\phi_W([-ab,0]) = [-a'b',0]$, $\phi_W([-ac,0]) = [-a'c',0]$, and $\phi_W([-bc,0]) = [-b'c',0]$. Note that the triple $\{a',b',c'\}$ is determined only up to the sign.

\section{Another view on families $\mathcal{F}_i$}
We start with elements of the family $\mathcal{F}_1$. Let $\{a,b,c\}$ be a special rational Diophantine triple ($a^2+1$ and $b^2+1$ are perfect squares and $S(a,b,c)=0$) for which $r_5(a,a,b,b,c)=0$ (i.e. $(a,a,b,b,c)$ is a regular quintuple). Let $A,B\in E_{a,b,c}(\Q)$ for which $x(A)=a \cdot abc$ and $x(B)=b\cdot abc$ (these points are rational since $\{a,b\}$ is a strong pair). Proposition \ref{prop:regularity} implies that the regularity condition is equivalent to $A\pm B \pm S=\mathcal{O}$ for some choice of sign. We can choose $A, B$ and $S$ so that $A+B+S=\mathcal{O}$ (recall that $S$ is a point of order $3$ with $x(S)=1$). Let $W_1=A+T_3$ and $W_2=B+T_2$, where $T_2=[-ac,0]$ and $T_3=[-bc,0]$ are the points of order $2$.

It follows from the following result (Proposition 4 in \cite{DKMS}) that $1-x(W_1)$ and $1-x(W_2)$ are perfect squares.

\begin{proposition}\label{prop:imrn}
	Let $Q$, $T$ and $[0,\alpha]$ be three rational points on an elliptic curve $\mathcal{E}$ over $\mathbb{Q}$
	given by the equation $y^2=f(x)$, where $f$ is a monic polynomial of degree $3$.
	Assume that $\mathcal{O} \not\in \{Q,T,Q+T\}$.
	Then
	$$ x(Q)x(T)x(Q+T)+ \alpha^2 $$
	is a perfect square.
\end{proposition}

Indeed, for $\mathcal{E}=E_{a,b,c}$ we have that $$x(W_1)x(T_3)x(A)+(abc)^2=x(W_1)(-bc)a\cdot abc+(abc)^2=(abc)^2(1-x(W_1))$$ is a perfect square. Similarly, we obtain that $1-x(W_2)$ is a perfect square.

Let $\phi_{W_1}:E_{a,b,c} \rightarrow E_{a',b',c'}$ be an isomorphism from Proposition \ref{prop: phi} associated to the point $W_1$.
The following proposition implies that a rational Diophantine triple $\{a',b',c'\}$ is special, satisfying the regularity condition \eqref{eg:reg}, and thus belongs to the $\mathcal{F}_2$ family.

\begin{proposition} We have that $a'^2=a^2$ and $b'=\frac{x(\phi_{W_1}(B+T_3))}{a'b'c'}$.
\end{proposition}
\begin{proof}
It is easy to check that $x(W_1)=1-k^2$, where $k^2=\frac{(ab+1)(ac+1)}{a^2+1}$. Hence
\[
\begin{aligned}
	\phi_{W_1}([-ab,0]) &= \left[ -\frac{a (a-c)}{a c+1}, 0 \right], \\
	\phi_{W_1}([-ac,0]) &= \left[ -\frac{a (a-b)}{a b+1}, 0 \right], \\
	\phi_{W_1}([-bc,0]) &= \left[ -\frac{(a-b) (a-c)}{(a b+1) (a c+1)}, 0 \right].
\end{aligned}
\]
Since $-a'^2=\frac{x(\phi_{W_1}([-ab,0]))x(\phi_{W_1}([-ac,0]))}{x(\phi_{W_1}([-bc,0]))}$, it follows that $a'^2=a^2$.
The second statement follows from direct computation in MAGMA.

\end{proof}

It follows that $\{a',b'\}$ is a strong pair since $a'^2+1=a^2+1$ is a perfect square, and $b'^2+1$ is a perfect square since the point $B'= \phi_{W_1}(B+T_3)$, with $x(B')=b'\cdot a'b'c'$ is rational. Moreover,
\begin{align*}
	\mathcal{O} &= \phi_{W_1}(A+B+S) \\
	&= \phi_{W_1}(A+T_3) + \phi_{W_1}(B+T_3) + \phi_{W_1}(S) \\
	&= P' + B' + S',
\end{align*}
which, according to Proposition \ref{prop:regularity}, implies the regularity condition $r_4(a',b',b',c')=0$.

More precisely, through direct computation, we derive the following proposition.

\begin{proposition}
	Let $(u_0,v_0)\in \mathcal{C}(\Q)$ be a rational point on the curve $\mathcal{C}$, $[a,b,c]=\mathcal{F}_1(u_0,v_0)$ the corresponding Diophantine triple, and $W_1,W_2\in E_{a,b,c}(\Q)$ points defined as above. The triples associated to points $W_1$ and $W_2$ by Proposition \ref{prop: phi} are equal to $\mathcal{F}_2(u_0,v_0)$ and  $\mathcal{F}_3(u_0,v_0)$ respectively. 
	
	Similarly, if $[a,b,c]=\mathcal{F}_2(u_0,v_0)$  then the triples associated to points $W_1$ and $W_2$ are equal to $\mathcal{F}_1(u_0,v_0)$ and  $\mathcal{F}_3(u_0,v_0)$ respectively, and if  $[a,b,c]=\mathcal{F}_3(u_0,v_0)$  then the triples associated to points $W_1$ and $W_2$ are equal to $\mathcal{F}_1(u_0,v_0)$ and  $\mathcal{F}_2(u_0,v_0)$ respectively. 
\end{proposition}

\subsection*{Example}

We now go back to our starting numerical examples from Section \ref{sec:experiments}. Consider first a special rational Diophantine triple $\{a,b,c\}$ where $a=30464/2223$, $b=22815/5168$ and  $c=361/7956$. Note that $\{a,b,c\}=\mathcal{F}_1(u_0,v_0)$, where $(u_0,v_0)=(-119/128, -135/169)$ is a rational point on the curve $\mathcal{C}$. Consider the rational points
\[
A = [ 250880/6669, 94938136300/252028179 ],
\]
\[
 B = [ 266175/21964, 18177179755/170264928 ],
\]
on $E_{a,b,c}$ which correspond to the strong elements $a$ and $b$.
Let $S=[1,-3307949/302328]$ be a point of order $3$. The regularity condition $r_5(a,a,b,b,c)=0$ is then equivalent to $A+B+S=\mathcal{O}$. Let $W_1=A+[-bc,0]=[19824/42025,-726438832196/108524729625]$ and $W_2=B+[-ac,0]=[-64155/24649,29291888395/1764671208]$. When we apply Proposition \ref{prop: phi} to the points $W_1$ and $W_2$ (recall that $1-x(W_1)$ and $1-x(W_2)$ are perfect squares), using the isomorphisms $\phi_{W_1}$ and $\phi_{W_2}$ respectively, we obtain triples $\mathcal{F}_2(u_0,v_0)=\{ \frac{30464}{2223}, \frac{4807}{31824}, \frac{10881}{1292}\}$ and $\mathcal{F}_3(u_0,v_0)=\{ \frac{-22815}{5168}, \frac{4807}{31824}, \frac{-8092}{2223} \}$ from our introductory example.

{\bf Acknowledgements.}
The authors wish to thank Randall Rathbun for providing examples of Diophantine sextuples with one strong element, which served as inspiration for the exploration of this topic.
The authors were supported by the Croatian Science Foundation under the project no.~IP-2022-10-5008.
This work was supported by the project “Implementation of cutting-edge research and its application as part of the Scientific Center of Excellence for Quantum and Complex Systems, and Representations of Lie Algebras“, PK.1.1.02, European Union, European Regional Development Fund.

\end{document}